\documentclass{amsart}%
\usepackage{eurosym}
\usepackage{amsfonts}
\usepackage{amsmath}
\usepackage{amssymb}
\usepackage{graphicx}%
\providecommand{\U}[1]{\protect\rule{.1in}{.1in}}
\newtheorem{theorem}{Theorem}
\theoremstyle{plain}

\newtheorem{definition}{Definition}
\newtheorem{example}{Example}

\newtheorem{lemma}{Lemma}

\newtheorem{proposition}{Proposition}

\numberwithin{equation}{section}
\begin{document}
\title{New extensions of Popoviciu's inequality}
\author{Marcela V. Mihai}
\address{Romanian Mathematical Society, Academy Street no. 14, RO-010014, Bucharest, Romania.}
\email{mmihai58@yahoo.com}
\author{Flavia-Corina Mitroi-Symeonidis}
\address{Faculty of Engineering Sciences, University of South-East Europe - LUMINA,
\c{S}os. Colentina 64b, Bucharest, RO-021187, Romania}
\email{fcmitroi@yahoo.com }
\dedicatory{Dedicated to the memory of T. Popoviciu.}
\begin{abstract}
Popoviciu's inequality is extended to the framework of $h$-convexity and also
to convexity with respect to a pair of quasi-arithmetic means. Several
applications are included.

\end{abstract}
\maketitle

\section{Introduction}

Fifty years ago Tiberiu Popoviciu \cite{Pop1965} published the following
striking characterization of convex functions:

\begin{theorem}
\label{thmPop}A real-valued continuous function $f$ defined on an interval $I$
is convex if and only if it verifies the inequality%
\begin{multline}
\frac{f(x)+f(y)+f(z)}{3}+f\left(  \frac{x+y+z}{3}\right) \tag{$Pop$}\\
\qquad\geq\frac{2}{3}\left(  f\left(  \frac{x+y}{2}\right)  +f\left(
\frac{y+z}{2}\right)  +f\left(  \frac{z+x}{2}\right)  \right) \nonumber
\end{multline}
whenever $x,y,z\in I$.
\end{theorem}

He also noticed that inequality ($Pop$) has higher order analogues for each
finite string of points (of length greater than or equal to 3). In
\cite{Pop1965}, only the unweighted case was discussed, but Popoviciu's
argument covers the weighted case as well.

Popoviciu's result has received a great deal of attention and many
improvements and extensions have been obtained. The interested reader may
consult the books of Mitrinovi\'{c} \cite{Mit1970}, Niculescu and Persson
\cite{NPop2006} and Pe\v{c}ari\'{c}, Proschan and Tong \cite{PPT}, as well as
the recent papers by Niculescu and his collaborators \cite{BNP2010},
\cite{KNP}, \cite{N2009}, \cite{NPop2006}, \cite{NPop2006a}, \cite{NR2015A},
\cite{NR2015MIA} and \cite{NSt}.

Two easy extensions of Popoviciu's inequality that escaped unnoticed refer to
the case of convex functions with values in a Banach lattice and that of
semiconvex functions (i.e., of the functions that become convex after the
addition of a suitable smooth function).\ Using the phenomenon of
semiconvexity one can state a Popoviciu type inequality for all functions of
class $C^{2}:$

\begin{proposition}
\label{propscPop}Suppose that $f\in C^{2}\left(  \left[  a,b\right]  \right)
$ and put
\[
M=\sup\left\{  f^{\prime\prime}(x):x\in\lbrack a,b]\right\}  \text{ and
}m=\inf\left\{  f^{\prime\prime}(x):x\in\lbrack a,b]\right\}  .
\]
Then%
\begin{multline*}
\frac{M}{36}\left(  (x-y)^{2}+(y-z)^{2}+(z-x)^{2}\right)  \geq\\
\frac{f(x)+f(y)+f(z)}{3}+f\left(  \frac{x+y+z}{3}\right)  -\frac{2}{3}\left(
f\left(  \frac{x+y}{2}\right)  +f\left(  \frac{y+z}{2}\right)  +f\left(
\frac{z+x}{2}\right)  \right) \\
\geq\frac{m}{36}\left(  (x-y)^{2}+(y-z)^{2}+(z-x)^{2}\right)
\end{multline*}
for all $x,y,z\in\lbrack a,b]$.
\end{proposition}

Indeed, under the assumptions of Proposition \ref{propscPop}, both functions
$\frac{M}{2}x^{2}-f(x)$ and $f(x)-\frac{m}{2}x^{2}$ are convex and Theorem
\ref{thmPop} applies. The variant of Proposition \ref{propscPop} for strongly
convex functions, that is for those functions $f$ such that $f-\frac{C}%
{2}x^{2}$ is convex for a suitable $C>0)$ can be deduced in the same manner:%
\begin{multline*}
\frac{f(x)+f(y)+f(z)}{3}+f\left(  \frac{x+y+z}{3}\right)  -\frac{2}{3}\left(
f\left(  \frac{x+y}{2}\right)  +f\left(  \frac{y+z}{2}\right)  +f\left(
\frac{z+x}{2}\right)  \right) \\
\geq\frac{C}{36}\left(  (x-y)^{2}+(y-z)^{2}+(z-x)^{2}\right)  .
\end{multline*}

Since $e^{x}\geq\frac{1}{2}x^{2}$ for $x\geq0,$ this fact yields the
inequality
\[
\frac{a+b+c}{3}+\sqrt[3]{abc}-\frac{2}{3}\left(  \sqrt{ab}+\sqrt{bc}+\sqrt
{ca}\right)  \geq\frac{1}{36}\left(  \log^{2}\frac{a}{b}+\log^{2}\frac{b}%
{c}+\log^{2}\frac{c}{a}\right)  ,
\]
for all $a,b,c\geq1.$

The aim of the present paper is to discuss Popoviciu's inequality in the
context of generalized convexity.

The next section deals with the case of convexity with respect to a pair of
means. See Definition \ref{defMNconv} below for details. Theorem
\ref{thmPopcvasiar} states the analogue of Popoviciu's inequality in the
context of quasi-arithmetic means, and its usefulness is illustrated by the
case of the hypergeometric function and the volume function of the unit ball
in $L^{P}$ spaces of dimension $n.$ A counter-example shows that we cannot
expect a full extension of Popoviciu's inequality to the case of arbitrary
convex functions with respect to a pair of means.

Section 3 deals with the case of $h$-convex functions in the sense of
Varo\v{s}anec \cite{V2007}. We end our paper by noticing the availability of
Popoviciu's inequality in the general framework of $h$-Jensen pairs of functions.

\section{The case of convex functions relative to a pair of means}

Convexity relative to a pair of means was first considered by Aumann
\cite{A1933} in 1933, but its serious investigation started not until the 90s.
By a \emph{mean} on an interval $I$ we understand any function $M:I\times
I\rightarrow\mathbb{R}$ such that%
\[
\min\left\{  x,y\right\}  \leq M(x,y)\leq\max\left\{  x,y\right\}
\]
for all $x,y\in I$. The most used class of means is that of quasi-arithmetic
means, which are associated to a continuous and strictly monotonic function
$\varphi:I\rightarrow\mathbb{R}$ by the formula
\[
\mathfrak{M}_{\varphi}(x,y)=\varphi^{-1}\left(  \frac{\varphi(x)+\varphi
(y)}{2}\right)  ,\text{ for }x,y\in I.
\]
A particular case is that of power means of order $p\in\mathbb{R},$%
\[
M_{p}(x,y)=\left\{
\begin{array}
[c]{rl}%
\min\left\{  x,y\right\}  & \text{if }p=-\infty\\
\left(  \frac{x^{p}+y^{p}}{2}\right)  ^{1/p} & \text{if }p\neq0\\
\sqrt{xy} & \text{if }p=0\\
\max\left\{  x,y\right\}  & \text{if }p=\infty,
\end{array}
\right.
\]
which corresponds to the function $\varphi(x)=x^{p}$, if $p\in\mathbb{R}%
\backslash\left\{  0\right\}  $ and $\varphi(x)=\log x$, if $p=0.$ Notice
that
\begin{align*}
M_{-1}  &  =H\text{ (the harmonic mean)}\\
M_{0}  &  =G\text{ (the geometric mean)}\\
M_{1}  &  =A\text{ (the arithmetic mean).}%
\end{align*}

Remarkably, the quasi-arithmetic means $\mathfrak{M}_{\varphi}$ admit natural
extensions to the case of an arbitrary finite family of points $x_{1}%
,...,x_{n}$ endowed with weights $\lambda_{1},...,\lambda_{n}$ of total mass
1,
\[
\mathfrak{M}_{\varphi}(x_{1},...,x_{n};\lambda_{1},...,\lambda_{n}%
)=\varphi^{-1}\left(  \sum_{k=1}^{n}\lambda_{k}\varphi(x_{k})\right)  .
\]

In order to simplify the notation, we put $\mathfrak{M}_{\varphi}%
(x_{1},...,x_{n};1/n,...,1/n)=\mathfrak{M}_{\varphi}(x_{1},...,x_{n}).$

\begin{definition}
\label{defMNconv}Given a pair of intervals $I$ and $J$ endowed respectively
with the means $M$ and $N,$ a function $f:I\rightarrow J$ is called
$(M,N)$-\emph{convex} if it is continuous and%
\begin{equation}
f\left(  M(x,y)\right)  \leq N\left(  f(x),f(y)\right)  \text{\quad for all
}x,y\in I. \tag{(M,N)}%
\end{equation}

\end{definition}

The analogue of Jensen's inequality works in the case of $\left(
\mathfrak{M}_{\varphi},\mathfrak{M}_{\psi}\right)  $-convex functions, so that
for such functions we have%
\[
f\left(  \mathfrak{M}_{\varphi}(x_{1},...,x_{n};\lambda_{1},...,\lambda
_{n})\right)  \leq\mathfrak{M}_{\psi}\left(  f(x_{1}),...,f(x_{n});\lambda
_{1},...,\lambda_{n}\right)
\]
for all $x_{1},...,x_{n}\in I$ and $\lambda_{1},...,\lambda_{n}\in\lbrack0,1]$
with $\sum\lambda_{k}=1.$

Clearly, the usual convex functions represent the case of $\left(  A,A\right)
$-convex functions, while the log-convex functions are the same with $\left(
A,G\right)  $-convex functions.

The importance and significance of other classes of generalized convex
functions such as of $\left(  G,A\right)  $-convex functions, $\left(
G,G\right)  $-convex functions, $(H,A)$-convex functions etc. is discussed in
the book \cite{NP2006} and the paper of Anderson, M.K. Vamanamurthy, M.
Vuorinen \cite{AVV}.

Not all important means are quasi-arithmetic. Two examples are the
\emph{logarithmic mean},%
\[
L\left(  a,b\right)  =\left\{
\begin{array}
[c]{cr}%
\frac{a-b}{\ln a-\ln b} & \text{if }a\neq b\\
a & \text{if }a=b
\end{array}
\right.
\]
and the \emph{identric mean,}%
\[
I\left(  a,b\right)  =\left\{
\begin{array}
[c]{cl}%
\frac{1}{e}\left(  \frac{b^{b}}{a^{a}}\right)  ^{\frac{1}{b-a}} & \text{if
}a\neq b\\
a & \text{if }a=b.
\end{array}
\right.
\]

The theory of $\left(  \mathfrak{M}_{\varphi},\mathfrak{M}_{\psi}\right)
$-convex functions can be deduced from the theory of usual convex functions.

\begin{lemma}
\label{ThmAczel}\emph{(J. }Aczel \emph{\cite{Ac1947})}. Let $\varphi\text{ and
}\psi$ be two strictly monotonic functions defined respectively on the
intervals $I$ and $J$, and let $f:I\rightarrow J$ be an arbitrary function.

If $\psi$ is strictly increasing, then $f$ is $\left(  \mathfrak{M}_{\varphi
},\mathfrak{M}_{\psi}\right)  $-convex/concave if and only if $\psi\circ
f\circ\varphi^{-1}$ is convex/concave on $\varphi\left(  I\right)  $ in the
usual sense.

If $\psi$ is strictly decreasing, then $f$ is $\left(  \mathfrak{M}_{\varphi
},\mathfrak{M}_{\psi}\right)  $-convex/concave if and only if $\psi\circ
f\circ\varphi^{-1}$ is concave/convex on $\varphi\left(  I\right)  $ in the
usual sense.
\end{lemma}

Theorem \ref{ThmAczel} yields the following extension of Popoviciu's inequality:

\begin{theorem}
\label{thmPopcvasiar}Suppose that $f:I\rightarrow J$ is an $\left(
\mathfrak{M}_{\varphi},\mathfrak{M}_{\psi}\right)  $-convex function. If
$\psi$ is strictly increasing, then
\begin{multline*}
\mathfrak{M}_{\psi}\left(  \mathfrak{M}_{\psi}\left(  f(x),f(y),f(z)\right)
,f\left(  \mathfrak{M}_{\varphi}\left(  x,y,z\right)  \right)  \right) \\
\geq\mathfrak{M}_{\psi}\left(  f\left(  \mathfrak{M}_{\varphi}\left(
x,y\right)  \right)  ,f\left(  \mathfrak{M}_{\varphi}\left(  y,z\right)
\right)  ,f\left(  \mathfrak{M}_{\varphi}\left(  z,x\right)  \right)  \right)
\end{multline*}
for all $x,y,z\in I$.

The inequality works in the reverse sense if the function $\psi$ is strictly decreasing.
\end{theorem}

\begin{proof}
By Lemma \ref{ThmAczel} the function $\psi\circ f\circ\varphi^{-1}$ is convex
on the interval $\varphi(I)$ so that one can apply Popoviciu's inequality to
it relative to the points $a=\varphi(x),$ $b=\varphi(y)$ and $c=\varphi(z).$
Then%
\begin{multline*}
\frac{\left(  \psi\circ f\circ\varphi^{-1}\right)  (\varphi(x))+\left(
\psi\circ f\circ\varphi^{-1}\right)  (\varphi(y))+\left(  \psi\circ
f\circ\varphi^{-1}\right)  (\varphi(z))}{3}\\
+\left(  \psi\circ f\circ\varphi^{-1}\right)  \left(  \frac{\varphi
(x)+\varphi(y)+\varphi(z)}{3}\right) \\
\qquad\geq\frac{2}{3}\left(  \left(  \psi\circ f\circ\varphi^{-1}\right)
\left(  \frac{\varphi(x)+\varphi(y)}{2}\right)  +\left(  \psi\circ
f\circ\varphi^{-1}\right)  \left(  \frac{\varphi(y)+\varphi(z)}{2}\right)
\right. \\
+\left.  \left(  \psi\circ f\circ\varphi^{-1}\right)  \left(  \frac
{\varphi(z)+\varphi(x)}{2}\right)  \right)  ,
\end{multline*}
that is,%
\begin{multline*}
\frac{1}{2}\left(  \frac{\psi\left(  f(x)\right)  +\psi\left(  f(y)\right)
+\psi\left(  f(z)\right)  }{3}+\psi\left(  f\left(  \mathfrak{M}_{\varphi
}\left(  x,y,z\right)  \right)  \right)  \right) \\
\geq\frac{\psi\left(  f\left(  \mathfrak{M}_{\varphi}\left(  x,y\right)
\right)  \right)  +\psi\left(  f\left(  \mathfrak{M}_{\varphi}\left(
y,z\right)  \right)  \right)  +\psi\left(  f\left(  \mathfrak{M}_{\varphi
}\left(  z,x\right)  \right)  \right)  }{3}.
\end{multline*}
On the other hand,%
\begin{align*}
\frac{\psi\left(  f(x)\right)  +\psi\left(  f(y)\right)  +\psi\left(
f(z)\right)  }{3}  &  =\psi\left(  \psi^{-1}\left(  \frac{\psi\left(
f(x)\right)  +\psi\left(  f(y)\right)  +\psi\left(  f(z)\right)  }{3}\right)
\right) \\
&  =\psi\left(  \mathfrak{M}_{\psi}\left(  f(x),f(y),f(z)\right)  \right)  ,
\end{align*}
and the proof ends by applying $\psi^{-1}$ to both sides.
\end{proof}

This result can be extended to the case of an arbitrary finite family of
points and weighted quasi-arithmetic means, but the details are tedious and
will be omitted.

\begin{example}
The Gaussian hypergeometric function (of parameters $a,b,c>0$) is defined via
the formula%
\[
F(x)=_{2}F_{1}(x;a,b,c)=\sum_{n=0}^{\infty}\frac{\left(  a,n\right)
(b,n)}{(c,n)n!}x^{n}\,\quad\text{for }\left\vert x\right\vert <1,
\]
where $(a,n)=a(a+1)\cdots(a+n-1)$ if $n\geq1$ and $(a,0)=1$. Anderson,
Vamanamurthy and Vuorinen \cite{AVV} proved that if $a+b\geq c>2ab$ and $c\geq
a+b-1/2$, then the function $1/F(x)$ is concave on $(0,1)$. This implies%
\[
F\left(  \frac{x+y}{2}\right)  \leq\frac{1}{\frac{1}{2}\left(  \frac{1}%
{F(x)}+\frac{1}{F(y)}\right)  }\text{\quad for all }x,y\in(0,1),
\]
whence it follows that the hypergeometric function is $\left(  A,H\right)
$-convex. Taking into account that the harmonic mean is a quasi-arithmetic
mean corresponding to the strictly decreasing function $\frac{1}{x},$ we infer
that $F$ verifies the following analogue of Popoviciu's inequality:
\[
\frac{1}{\frac{1}{2}\left(  \frac{1}{3}\left(  \frac{1}{F(x)}+\frac{1}%
{F(y)}+\frac{1}{F(z)}\right)  +\frac{1}{F\left(  \frac{x+y+z}{3}\right)
}\right)  }\leq\frac{1}{\frac{1}{3}\left(  \frac{1}{F\left(  \frac{x+y}%
{2}\right)  }+\frac{1}{F\left(  \frac{y+z}{2}\right)  }+\frac{1}{F\left(
\frac{z+x}{2}\right)  }\right)  },
\]
equivalently,%
\begin{multline*}
\frac{1}{2}\left(  \frac{1}{3}\left(  \frac{1}{F(x)}+\frac{1}{F(y)}+\frac
{1}{F(z)}\right)  +\frac{1}{F\left(  \frac{x+y+z}{3}\right)  }\right) \\
\geq\frac{1}{3}\left(  \frac{1}{F\left(  \frac{x+y}{2}\right)  }+\frac
{1}{F\left(  \frac{y+z}{2}\right)  }+\frac{1}{F\left(  \frac{z+x}{2}\right)
}\right)  .
\end{multline*}

\end{example}

\begin{example}
D. Borwein, J. Borwein, G. Fee and R. Girgensohn \cite{BBFG} proved that the
volume $V_{n}(p)$ of the convex body $\mathcal{E}=\left\{  x\in\mathbb{R}%
^{n}:\left\Vert x\right\Vert _{p}\leq1\right\}  $ is an $(H,G)$-concave
function on $[1,\infty)$. More precisely, given $\alpha>1,$ the function
\[
V_{\alpha}(p)=2^{\alpha}\frac{\Gamma^{\alpha}(1+1/p)}{\Gamma\left(
1+\alpha/p\right)  }%
\]
verifies the inequality%
\[
V_{\alpha}^{1-\lambda}(p)V_{\alpha}^{\lambda}(q)\leq V_{\alpha}\left(
\frac{1}{\frac{1-\lambda}{p}+\frac{\lambda}{q}}\right)
\]
for all $p,q>0$ and $\lambda\in\lbrack0,1].$\ In this case, Popoviciu's
inequality becomes%
\begin{multline*}
\sqrt{\sqrt[3]{V_{\alpha}\left(  p\right)  V_{\alpha}\left(  q\right)
V_{\alpha}\left(  r\right)  }\cdot V_{\alpha}\left(  \dfrac{1}{\frac{1}%
{3}\left(  \frac{1}{p}+\frac{1}{q}+\frac{1}{r}\right)  }\right)  }\\
\geq\sqrt[3]{V_{\alpha}\left(  \dfrac{1}{\frac{1}{2}\left(  \frac{1}{p}%
+\frac{1}{q}\right)  }\right)  \cdot V_{\alpha}\left(  \dfrac{1}{\frac{1}%
{2}\left(  \frac{1}{q}+\frac{1}{r}\right)  }\right)  \cdot V_{\alpha}\left(
\dfrac{1}{\frac{1}{2}\left(  \frac{1}{r}+\frac{1}{p}\right)  }\right)  }.
\end{multline*}

\end{example}

A natural question is whether Popoviciu's inequality works for an arbitrary
$(M,N)$-convex function.

We shall see that the answer is negative. Indeed, the log-convex functions are
also $\left(  A,L\right)  $-convex, because they verify the inequalities%
\begin{multline*}
f\left(  \frac{a+b}{2}\right) \\
\leq\exp\left(  \frac{1}{b-a}\int_{a\,}^{b}\log\,f(x)\,\mathrm{d}x\right)
\leq\frac{1}{b-a}\,\int_{a}^{b}\,f(x)\,\mathrm{d}x\leq L(f(a),f(b))\\
\leq\frac{f(a)+f(b)}{2}.
\end{multline*}
See \cite{N2012}.

The logarithmic mean was extended to the case of an arbitrary finite family of
points by Neuman in his paper \cite{EdN1994}. An argument that Neuman's
extension is the "right" one can be found in \cite{N2012}. For triplets, the
logarithmic mean is given by the formula%
\[
L(a,b,c)=\frac{2a}{\log\frac{a}{b}\log\frac{a}{c}}+\frac{2b}{\log\frac{b}%
{a}\log\frac{b}{c}}+\frac{2c}{\log\frac{c}{a}\log\frac{c}{b}}.
\]

The analogue of Popoviciu's inequality in the case of $\left(  A,L\right)
$-convex functions should be
\begin{multline*}
\frac{L\left(  f(x),f(y),f(z)\right)  -f\left(  \frac{x+y+z}{3}\right)  }{\log
L\left(  f(x),f(y),f(z)\right)  -\log f\left(  \frac{x+y+z}{3}\right)  }\\
\geq L\left(  f\left(  \frac{x+y}{2}\right)  ,f\left(  \frac{y+z}{2}\right)
,f\left(  \frac{z+x}{2}\right)  \right)  ,
\end{multline*}
for all $x,y,z$ belonging to the domain of $f.$ However this does not work
even in the case of the Gamma function,
\[
\Gamma(x)=\int_{0}^{\infty}t^{x-1}e^{-t}dt,\quad x>0,
\]
that is known to be log-convex (see \cite{NP2006}, Theorem 2.2.1, pp. 68-69).
The Gamma function has a minimum at 1.461632..., so we will search around this point.

Put%
\begin{multline*}
E(x;y;z)=\frac{L\left(  \Gamma(x),\Gamma(y),\Gamma(z)\right)  -\Gamma\left(
\frac{x+y+z}{3}\right)  }{\log L\left(  \Gamma(x),\Gamma(y),\Gamma(z)\right)
-\log\Gamma\left(  \frac{x+y+z}{3}\right)  }\\
-L\left(  \Gamma\left(  \frac{x+y}{2}\right)  ,\Gamma\left(  \frac{y+z}%
{2}\right)  ,\Gamma\left(  \frac{z+x}{2}\right)  \right)
\end{multline*}
for $x,y,z>0.$ A simple computation shows that%
\[
E(1.40;1.46;1.47)=65.92090117-108.64<0
\]
while%
\[
E(0.30;0.34;0.35)=2.711369453-2.709270>0.
\]
Therefore Popoviciu's inequality does not always work for $\left(  M,N\right)
$-convex functions.

\section{The case of h-convex functions}

In 2007, Varo\v{s}anec \cite{V2007} introduced a class of generalized convex
functions that brings together several important classes of functions.

In order to enter into the details we have to fix a function
$h:(0,1)\rightarrow(0,\infty)$ such that
\begin{equation}
h(1-\lambda)+h(\lambda)\geq1\text{ for all }\lambda\in(0,1). \label{h1}%
\end{equation}
As above, $I$ will denote an interval.

\begin{definition}
\label{defhconv}A function $f:I\rightarrow\mathbb{R}$ is called $h$-convex if%
\[
f\left(  \left(  1-\lambda\right)  x+\lambda y\right)  \leq h(1-\lambda
)f(x)+h(\lambda)f(y)
\]
for all $x,y\in I$ and $\lambda\in(0,1).$
\end{definition}

The role of the condition (\ref{h1}) is to assure that the function
identically 1 is $h$-convex.

The usual convex functions represent the particular case of Definition
\ref{defhconv}, where $h$ is the identity function$.$

The $h$-convex functions corresponding to the case $h(\lambda)=\lambda^{s}$
(for a suitable $s\in(0,1])$ are the $s$-\emph{convex functions} in the sense
of Breckner \cite{B1978}. Their systematic study can be found in the papers of
Hudzik and Maligranda \cite{HM94} and Pinheiro \cite{P2007}.

An example of an $s$-convex function (for $0<s<1)$ is given by the formula%
\[
f(t)=\left\{
\begin{array}
[c]{cl}%
a & \text{if }t=0\\
bt^{s}+c & \text{if }t>0
\end{array}
\right.
\]
where $b\geq0$ and $0\leq c\leq a.$ In particular, the function $t^{s}$ is
$s$-convex on $[0,\infty)$ if $0<s<1.$

The nonnegative $h$-convex functions corresponding to the case $h(\lambda
)=\frac{1}{\lambda}$ are the \emph{convex functions in the sense of
Godunova-Levin} \cite{GL1985}. They verify the inequality%
\[
f\left(  \left(  1-\lambda\right)  x+\lambda y\right)  \leq\frac
{f(x)}{1-\lambda}+\frac{f(y)}{\lambda}%
\]
for all $x,y\in I$ and $\lambda\in(0,1).$ Every nonnegative monotonic function
(as well as every nonnegative convex function) is convex in the sense of Godunova-Levin.

The $h$-convex functions corresponding to the case $h(\lambda)\equiv1$ are
the\emph{ }$P$-\emph{convex functions} in the sense of\emph{ }Dragomir,
Pe\v{c}ari\'{c} and Persson \cite{DPP}. They verify inequalities of the form%
\[
f\left(  \left(  1-\lambda\right)  x+\lambda y\right)  \leq f(x)+f(y)
\]
for all $x,y\in I$ and $\lambda\in(0,1).$

One can state the following analogue of Popoviciu's inequality in the case of
$h$-convex functions.

\begin{theorem}
\label{thmhPop}If $h$ is concave, then every nonnegative $h$-convex function
$f\colon I\rightarrow\mathbb{R}$ verifies the inequality
\begin{multline}
\max\left\{  h\left(  1/2\right)  ,2h(1/4)\right\}  \left(
f(x)+f(y)+f(z)\right)  +2h(3/4)f\left(  \frac{x+y+z}{3}\right) \tag{$hPop$}\\
\qquad\geq f\left(  \frac{x+y}{2}\right)  +f\left(  \frac{y+z}{2}\right)
+f\left(  \frac{z+x}{2}\right) \nonumber
\end{multline}
for all $x,y,z\in I$.
\end{theorem}

\begin{proof}
Without loss of generality we may assume that $x\leq y\leq z$. If
$y\leq(x+y+z)/3$, then
\[
(x+y+z)/3\leq(x+z)/2\leq z\quad\text{and}\quad(x+y+z)/3\leq(y+z)/2\leq z,
\]
which yields two numbers $s,t\in\lbrack0,1]$ such that
\begin{gather*}
\frac{x+z}{2}=s\cdot\frac{x+y+z}{3}+(1-s)\cdot z\\
\frac{y+z}{2}=t\cdot\frac{x+y+z}{3}+(1-t)\cdot z.
\end{gather*}
Summing up, we get $(x+y-2z)(s+t-3/2)=0$. If $x+y-2z=0$, then necessarily
$x=y=z$, and the inequality $\left(  hPop\right)  $ is clear. If $s+t=3/2$, by
summing up the following three inequalities
\begin{align*}
f\left(  \frac{x+z}{2}\right)   &  \leq h(s)\cdot f\left(  \frac{x+y+z}%
{3}\right)  +h(1-s)\cdot f(z)\\
f\left(  \frac{y+z}{2}\right)   &  \leq h(t)\cdot f\left(  \frac{x+y+z}%
{3}\right)  +h(1-t)\cdot f(z)\\
f\left(  \frac{x+y}{2}\right)   &  \leq h\left(  1/2\right)  \cdot
f(x)+h\left(  1/2\right)  \cdot f(y).
\end{align*}
we get%
\begin{multline*}
f\left(  \frac{x+y}{2}\right)  +f\left(  \frac{y+z}{2}\right)  +f\left(
\frac{z+x}{2}\right) \\
\leq\left(  h(s)+h(t)\right)  \cdot f\left(  \frac{x+y+z}{3}\right) \\
+h\left(  1/2\right)  \cdot f(x)+h\left(  1/2\right)  \cdot f(y)+\left(
h(1-s)+h(1-t)\right)  \cdot f(z)\\
\leq h\left(  1/2\right)  \cdot f(x)+h\left(  1/2\right)  \cdot
f(y)+2h(1/4)\cdot f(z)+2h(3/4)f\left(  \frac{x+y+z}{3}\right) \\
\leq\max\left\{  h\left(  1/2\right)  ,2h(1/4)\right\}  \left(
f(x)+f(y)+f(z)\right)  +2h(3/4)f\left(  \frac{x+y+z}{3}\right)  ,
\end{multline*}
and the inequality $\left(  hPop\right)  $ is also clear.

The case where $(x+y+z)/3<y$ can be treated in a similar way.
\end{proof}

As an application of Theorem \ref{thmhPop} let us consider the case of the
function $t^{1/2}$ (which is $s$-convex for $s=1/2).$ Then $h(t)=t^{1/2},$
$\max\left\{  h\left(  1/2\right)  ,2h(1/4)\right\}  =1$ and $2h(3/4)=\sqrt
{3},$ which yields%
\begin{multline*}
x^{1/2}+y^{1/2}+z^{1/2}+\sqrt{3}\left(  \frac{x+y+z}{3}\right)  ^{1/2}\\
\geq\left(  \frac{x+y}{2}\right)  ^{1/2}+\left(  \frac{y+z}{2}\right)
^{1/2}+\left(  \frac{z+x}{2}\right)  ^{1/2}%
\end{multline*}
for all $x,y,z\geq0.$

We end our paper with another Popoviciu type inequality for $h$-convex functions.

The basic ingredient is the Jensen-type inequality for the $h$-convex
functions,
\[
f\left(  \frac{x_{1}+\cdots+x_{n}}{n}\right)  \leq h\left(  \frac{1}%
{n}\right)  \left(  f(x_{1})+\cdots+f(x_{n})\right)
\]
valid for arbitrary finite strings of points $x_{1},...,x_{n}$ under the
additional hypothesis that $h$ is supermultiplicative in the sense that
$h(xy)\geq h(x)h(y)$ for all $x,y.$ See \cite{V2007}, Theorem 19. When $f$ is
$h$-concave and $h$ is submultiplicative, the Jensen inequality takes the form%
\[
f\left(  \frac{x_{1}+\cdots+x_{n}}{n}\right)  \geq h\left(  \frac{1}%
{n}\right)  \left(  f(x_{1})+\cdots+f(x_{n})\right)  .
\]

\begin{theorem}
$i)$ If $h$ is supermultiplicative, with $h(1/3)<1,$ and $f:I\rightarrow
\mathbb{R}$ is an $h$-convex function, then
\begin{multline*}
f(x)+f(y)+f(z)-f\left(  \frac{x+y+z}{3}\right) \\
\geq\frac{1-h\left(  1/3\right)  }{2h\left(  1/2\right)  }\left(  f\left(
\frac{x+y}{2}\right)  +f\left(  \frac{y+z}{2}\right)  +f\left(  \frac{z+x}%
{2}\right)  \right)
\end{multline*}
for all $x,y,z\in I.$

$ii)$ If $h$ is submultiplicative, with $h(1/3)>1,$ and $f:I\rightarrow
\mathbb{R}$ is an $h$-concave function, then
\begin{multline*}
f\left(  \frac{x+y+z}{3}\right)  -\left(  f(x)+f(y)+f(z)\right) \\
\geq\frac{h\left(  1/3\right)  -1}{2h\left(  1/2\right)  }\left(  f\left(
\frac{x+y}{2}\right)  +f\left(  \frac{y+z}{2}\right)  +f\left(  \frac{z+x}%
{2}\right)  \right)  .
\end{multline*}

\end{theorem}

\begin{proof}
$i)$ In this case,
\begin{multline*}
f\left(  \frac{x+y}{2}\right)  +f\left(  \frac{y+z}{2}\right)  +f\left(
\frac{z+x}{2}\right)  \leq2h\left(  1/2\right)  \left(  f(x)+f(y)+f(z)\right)
\\
=\frac{2h\left(  1/2\right)  }{1-h\left(  1/3\right)  }\left(
f(x)+f(y)+f(z)\right)  -\frac{2h\left(  1/2\right)  }{1-h\left(  1/3\right)
}h\left(  1/3\right)  \left(  f(x)+f(y)+f(z)\right) \\
\leq\frac{2h\left(  1/2\right)  }{1-h\left(  1/3\right)  }\left(
f(x)+f(y)+f(z)\right)  -\frac{2h\left(  1/2\right)  }{1-h\left(  1/3\right)
}f\left(  \frac{x+y+z}{3}\right) \\
=\frac{2h\left(  1/2\right)  }{1-h\left(  1/3\right)  }\left(
f(x)+f(y)+f(z)-f\left(  \frac{x+y+z}{3}\right)  \right)  .
\end{multline*}

$ii)$ Similarly,%
\begin{multline*}
f\left(  \frac{x+y}{2}\right)  +f\left(  \frac{y+z}{2}\right)  +f\left(
\frac{z+x}{2}\right)  \geq2h\left(  1/2\right)  \left(  f(x)+f(y)+f(z)\right)
\\
=\frac{2h\left(  1/2\right)  h\left(  1/3\right)  }{h\left(  1/3\right)
-1}\left(  f(x)+f(y)+f(z)\right)  -\frac{2h\left(  1/2\right)  }{h\left(
1/3\right)  -1}\left(  f(x)+f(y)+f(z)\right) \\
\geq\frac{2h\left(  1/2\right)  }{h\left(  1/3\right)  -1}\left(  f\left(
\frac{x+y+z}{3}\right)  -\left(  f(x)+f(y)+f(z)\right)  \right)  .
\end{multline*}

\end{proof}

As an application of Theorem \ref{thmhPop} let us consider the case of the
function $t^{1/2}$ (which is $s$-convex for $s=1/2).$ Then $h(t)=t^{1/2}$ and
$h(1/3)=\left(  1/3\right)  ^{1/2}=\allowbreak0.577...<1.$ Therefore%
\begin{multline*}
x^{1/2}+y^{1/2}+z^{1/2}-\left(  \frac{x+y+z}{3}\right)  ^{1/2}\\
\geq\frac{1-\left(  1/3\right)  ^{1/2}}{2\left(  1/2\right)  ^{1/2}}\left[
\left(  \frac{x+y}{2}\right)  ^{1/2}+\left(  \frac{y+z}{2}\right)
^{1/2}+\left(  \frac{z+x}{2}\right)  ^{1/2}\right]
\end{multline*}
for all $x,y,z\geq0.$

Last but not least it is worth noticing that Popoviciu's inequality still
works in the more general context of $h$-Jensen pairs $(f,g)$. These pairs are
aimed to satisfy inequalities of the form
\[
f\left(  (1-\lambda)x+\lambda y\right)  \leq h(1-\lambda)g(x)+h(\lambda)g(y),
\]
for all $x,y\in I$ and $\lambda\in(0,1);$ here $I$ is a common domain of $f$
and $g.$ An inspection of the argument of Theorem \ref{thmhPop} easily yields
the following result.

\begin{theorem}
Let $h$ be concave and $(f,g)$ be an $h$-Jensen pair of positive functions
$f,g\colon I\rightarrow\mathbb{R}$. Then a Popoviciu type inequality holds:%
\begin{multline*}
\max\left\{  h\left(  1/2\right)  ,2h(1/4)\right\}  \left(
g(x)+g(y)+g(z)\right)  +2h(3/4)g\left(  \frac{x+y+z}{3}\right) \\
\qquad\geq f\left(  \frac{x+y}{2}\right)  +f\left(  \frac{y+z}{2}\right)
+f\left(  \frac{z+x}{2}\right)  .
\end{multline*}

\end{theorem}

\end{document}